\documentclass[final]{jmsj}
\usepackage{graphicx}

\usepackage{amssymb,amsmath,amsthm}
\usepackage[all]{xy}
\usepackage{xcolor,colortbl}
\usepackage{soul}

\newtheorem{theorem}{Theorem}[section]
\newtheorem{proposition}[theorem]{Proposition}
\newtheorem{lemma}[theorem]{Lemma}
\newtheorem{corollary}[theorem]{Corollary}
 
\theoremstyle{definition}

\newtheorem{remark}[theorem]{Remark}

\newtheorem*{theorem*}{\bf{Theorem}}

\numberwithin{equation}{section}

\newcommand{\sfD}{\mathsf{D}}

\newcommand{\bmu}{\boldsymbol{\mu}}
\newcommand{\balpha}{\boldsymbol{\alpha}}

\newcommand{\bbZ}{{\mathbb{Z}}}
\newcommand{\bbP}{{\mathbb{P}}}
\newcommand{\bbG}{{\mathbb{G}}}
\newcommand{\bbC}{{\mathbb{C}}}
\newcommand{\bbQ}{{\mathbb{Q}}}

\newcommand{\Aut}{\operatorname{Aut}}

\newcommand{\Pic}{\operatorname{Pic}}

\newcommand{\NS}{\operatorname{NS}}

\newcommand{\Ker}{\operatorname{Ker}}

\newcommand{\Num}{\operatorname{Num}}
\newcommand{\Or}{\operatorname{Or}}
\newcommand{\nt}{\operatorname{nt}}
\newcommand{\ct}{\operatorname{ct}}
\newcommand{\Tors}{\operatorname{Tors}}
\newcommand{\Hom}{\operatorname{Hom}}
\newcommand{\nd}{\operatorname{nd}}
\newcommand{\bfAut}{\mathbf{Aut}}
\newcommand{\nod}{\operatorname{nod}}
\newcommand{\rank}{\operatorname{rank}}
\newcommand{\ord}{\operatorname{ord}}

\newcommand{\calO}{\mathcal{O}}

\newcommand{\calL}{\mathcal{L}}

\newcommand{\et}{\operatorname{\acute{e}t}}
\newcommand{\beq}{\begin{equation}}
\newcommand{\eeq}{\end{equation}}
\usepackage[
            pdftex,
           colorlinks=true,
            urlcolor=blue,       
            filecolor=blue,     
            linkcolor=blue,       
            citecolor=blue,         
            pdftitle= {Title},
            pdfauthor={Author},
            pdfsubject={Subject},
            pdfkeywords={Key}
            pagebackref,
            bookmarksopen=true]
            {hyperref}
\usepackage{hyperref}

\author{Igor Dolgachev}
\address{\hfill \newline 
Department of Mathematics \newline
University of Michigan \newline
2072 East Hall  \newline
525 East University Avenue \newline
Ann Arbor,
 MI 48109-1109 USA}
\email{idolga@umich.edu}

\author{Gebhard Martin}
\address{\hfill \newline
Mathematisches Institut  \newline
Universit\"at Bonn \newline
Endenicher Allee 60 \newline
53115 Bonn \newline
Germany}
\email{gmartin@math.uni-bonn.de}

\date{}
\dedicatory{To Shigeyuki Kond\={o} on the occasion of his 60th birthday}
\title[Numerically trivial automorphisms of Enriques surfaces]{Numerically trivial automorphisms of Enriques surfaces in  characteristic $2$}
\thanks{The second author is supported by the DFG Sachbeihilfe LI 1906/3 - 1 "Automorphismen von Enriques Fl\"achen".}

\subjclass[2010]{Primary 14J28; Secondary 14J50}
\keywords{Enriques surfaces, Action of automorphisms on cohomology}

\begin{document}
\maketitle
\begin{abstract}
An automorphism of an algebraic surface $S$ is called cohomologically  (numerically) trivial if it acts identically on  the second cohomology  group (this group modulo torsion subgroup). Extending the results of S. Mukai and Y. Namikawa to arbitrary characteristic $p > 0$, we prove that the group of cohomologically trivial automorphisms $\Aut_{\ct}(S)$ of an Enriques surface $S$ is of order $\le 2$ if $S$ is not supersingular. If $p = 2$ and $S$ is supersingular, we show that $\Aut_{\ct}(S)$ is a cyclic group of odd order $n\in \{1,2,3,5,7,11\}$ or the quaternion group $Q_8$ of order $8$ and we describe explicitly all the exceptional cases.  If $K_S \neq 0$, we also prove that 
the group $\Aut_{\nt}(S)$ of  numerically trivial automorphisms is a subgroup of a cyclic group of order $\le 4$ unless $p = 2$, where $\Aut_{\nt}(S)$ is a subgroup of a $2$-elementary group of rank $\le 2$. \end{abstract}

 \section{Introduction}  Let $S$ be a  smooth projective algebraic surface over an algebraically closed field $\Bbbk$ of characteristic $p\ge 0$. 
 An automorphism $g$ of $S$ is called \emph{cohomologically trivial} (resp. \emph{numerically  trivial}) if it acts identically on the  flat cohomology $H^2(S,\bbZ_l(1))$ (resp. $H^2(S,\bbZ_l(1))$ modulo torsion). An easy example is an automorphism isotopic to the identity, i.e. one that belongs to the connected group of automorphisms that preserves an ample divisor class. When the latter group is trivial, such an automorphism exists very rarely. For example, over the field of complex numbers, $S$ must be either an elliptic surface with $q=p_g = 0$ or with $c_2 = 0$, or a surface of general type whose canonical linear system has a base point or its Chern classes satisfy $c_1^2 = 2c_2$ or $c_1^2 = 3c_2$ (see \cite{Peters}). In particular, a complex K3 surface does not admit non-trivial numerically trivial automorphisms, while  a complex  Enriques surface could have them. The first example of such an 
 automorphism of an Enriques surface was constructed by D. Lieberman in 1976 \cite{Lieberman}. Later, S. Mukai and Y. Namikawa were able to give a complete classification of possible groups of cohomologically or numerically trivial automorphisms of complex Enriques surfaces as well as the surfaces themselves on which such automorphisms could act \cite{MukaiNam},\cite{Mukai}

In the case of algebraic surfaces over a field of positive characteristic we know less.  However, we know, for example, that K3 surfaces do not admit any numerically trivial automorphisms by work of A. Ogus \cite{Ogus}, J. Keum \cite{Keum1} and J. Rizov \cite{Rizov}.

 This paper deals with the case when $S$ is an Enriques surface. One of the main tools of the Mukai-Namikawa classification is the Global Torelli Theorem for K3 covers of Enriques surfaces. The absence of these tools in the case of characteristic $p > 0$ requires different methods. A paper \cite{DolgachevNum} of the first author was the first attempt to extend the work of Mukai and Namikawa to this case. Although the main result of the paper is correct when $p \ne 2$, some arguments were not complete and the analysis of possible groups in characteristic $2$ was erroneous and far from giving a classification of possible groups. In fact, a recent work of T. Katsura, S. Kondo and the second author  that gives a complete classification of Enriques surfaces in characteristic $2$ with finite automorphism group reveals  many possible groups of numerically trivial automorphisms that were claimed to be excluded in the paper. The goal of this paper is to use  some new ideas to give a complete classification of groups of numerically and cohomologically trivial automorphisms in characteristic two. For completeness sake, we also use the new ideas to 
treat  the case $p\ne 2$.   
  
 We show that, if the characteristic is not equal to $2$, the main assertion of Mukai and Namikawa still holds: the group $\Aut_{\ct}(S)$ is of order $\le 2$ and the group $\Aut_{\nt}(S)$ is a cyclic group of order $\le 4$. 
 
 If $p = 2$, $K_S \ne 0$ ($S$ is called a classical Enriques surface in this case) and $S$ is not $E_8$-extra-special, then 
 $\Aut_{\ct}(S)$ is trivial unless $S$ is an  extra-special surface of type $\tilde{D}_8$. In this case, the automorphism group is of order $2$. The group  $\Aut_{\nt}(S)$ is a subgroup of the product of two cyclic groups of order $2$. 
 
 If $p = 2$, and $S$ is an ordinary Enriques surface (defined by  the action of the Frobenius on its cohomology), 
 then $\Aut_{\ct}(S) = \Aut_{\nt}(S)$ is of order less than or equal to $2$.
 
 Finally, if $p = 2$ and $S$ is a supersingular Enriques surface, we prove that $\Aut_{\ct}(S)$ is of order $\le 2$ unless $S$ is "very special":  We show that the only Enriques surfaces with a cohomologically trivial automorphism of odd order $\ne 3$ or more than one cohomologically trivial automorphism of even order are certain exceptional or extra-special surfaces with finite automorphism group and we give some necessary conditions for a surface to have a cohomologically trivial automorphism of order $3$. 
 
The restrictions obtained on the possible groups of cohomologically and numerically trivial automorphisms are summarized in the following theorem.
 
  \begin{theorem*}
Let $S$ be an Enriques surface over an algebraically closed field of characteristic $p \geq 0$.
\begin{enumerate}
\item If $p \neq 2$, then $|\Aut_{\ct}(S)| \leq 2$ and $\Aut_{\nt}(S) \cong \bbZ/2^a\bbZ$ with $a \leq 2$.
\item If $p = 2$ and $S$ is ordinary, then $|\Aut_{\ct}(S)| = |\Aut_{\nt}(S)| \leq 2$.
\item If $p = 2$ and $S$ is classical and not $E_8$-extra-special, then $|\Aut_{\ct}(S)| \leq 2$ and $\Aut_{\nt}(S) \cong (\bbZ/2\bbZ)^a$ with $a \leq 2$.
\item If $p = 2$ and $S$ is supersingular, then $|\Aut_{\ct}(S)| = |\Aut_{\nt}(S)| \leq 2$, unless $S$ is one of five types of exceptions distinguished by their dual graphs of $(-2)$-curves.
\end{enumerate}
Moreover, if $S$ is unnodal, then $\Aut_{\ct}(S) = \{1\}$.
 \end{theorem*}

The proof of the above results will make use of \emph{bielliptic maps}, which will be recalled in Section $5$. Before this, in Section $2-4$, we give the necessary background material on numerically trivial automorphisms, on genus one curves and on genus one fibrations of Enriques surfaces. After explaining the classification of \emph{extra-special} Enriques surfaces in Section $6$, we prove our main results in Sections $7$ and $8$.
%
 
 \medskip
\noindent
{\bf Acknowledgement.}
The  authors thank T. Katsura and S. Kond\={o} for many interesting discussions on the subject. Moreover, the second author would like to thank S. Kond\={o} and the Department of Mathematics of Nagoya University for their kind hospitality during his stay there.

\section{Generalities on numerically and cohomologically trivial automorphisms}
Let $S$ be an Enriques surface. It is known that
$$H_{\et}^2(S,\bbZ_l) \cong \NS(S)\otimes \bbZ_l, \quad H^2_{\et}(S,\bbZ_l)/\textrm{torsion}\cong \Num(S)\otimes \bbZ_l,$$
where $\Num(S) =  \NS(S)/(K_S)$ is the group of divisor classes modulo numerical equivalence and $\NS(S)$ is the N\'eron-Severi group that coincides with the Picard group of $S$ (see \cite{CD}, Chapter 1, \S 2). The automorphism group $\Aut(S)$ is discrete in the sense that the connected component of the identity of the scheme of automorphisms $\bfAut_{S/\Bbbk}$ of $S$ consists of one point, and admits natural representations
$$\rho:\Aut(S)\to \Or(\NS(S)), \quad \rho_n:\Aut(S)\to \Or(\Num(S)),$$
in the group of  automorphisms of the corresponding abelian groups preserving the intersection form. We set
$$\Aut_{\ct}(S) = \Ker(\rho), \quad \Aut_{\nt}(S) = \Ker(\rho_n).$$
An automorphism in $\Ker(\rho)$ (resp. $\Ker(\rho_n)$) is called \emph{cohomologically trivial} (resp. \emph{numerically trivial}).

We start with the following general result that applies to any  surface with 
discrete scheme of automorphisms and discrete Picard scheme.

\begin{proposition}\label{P7.2.1} The groups $\Aut_{\ct}(S)$ and $\Aut_{\nt}(S)$ are finite groups.
\end{proposition} 

\begin{proof} We know that $\NS(S) = \Pic(S)$ and $\Num(S)$ is 
the quotient of 
$\NS(S)$ by its finite torsion subgroup  $\Tors(\NS(S))$. Thus, the elementary theory of abelian groups gives us 
$$
\Or(\NS(S)) \cong \Hom(\Num(S),\Tors(\NS(S)))\rtimes \Or(\Num(S)).
$$
This implies that 
\beq\label{r11}
\Aut_{\nt}(S)/\Aut_{\ct}(S) \subseteq \Tors(\NS(S))^{\oplus \rho(S)}.
\eeq 
So, it is enough to prove that $G = \Aut_{\ct}(S)$ is a finite group. The group acts trivially on $\Pic(S)$, hence leaves invariant any very 
ample invertible sheaf $\calL$. For any $g\in G$ let $\alpha_g:g^*(\calL)\to \calL$ be an isomorphism. Define a structure of a group on the 
set $\tilde{G}$ of pairs $(g,\alpha_g)$ by 
$$(g,\alpha_g)\circ (g',\alpha_{g'}) = (g\circ g', \alpha_{g'}\circ g'{}^*(\alpha_g)).$$
The homomorphism $(g,\alpha_g)\to g$ defines an isomorphism $\tilde{G}\cong \Bbbk^*\rtimes G$. 
The sheaf $\calL$ admits a natural $\tilde{G}$-linearization, and hence the group $\tilde{G}$ acts linearly on the space $H^0(S,\calL)$ and the action defines  an injective
homomorphism $G\to \Aut(\bbP(H^0(S,\calL))$. The group of projective transformations of $S$ embedded by $|\calL|$ is a linear algebraic group that has finitely many connected components. We know that $G$ is discrete. Thus, the group $G$ is finite.
\end{proof}

In our case, when $S$ is an Enriques surface, we know that the torsion subgroup of $\NS(S)$ is generated by the canonical class $K_S$ and  $2K_S = 0$. Moreover,   $K_S\ne 0$ if $p\ne 2$. Recall that, in characteristic $2$, Enriques surfaces come in three types:
\begin{itemize}
\item classical  surfaces,
\item ordinary Enriques surfaces or $\bmu_2$-surfaces,
\item supersingular surfaces or $\balpha_2$-surfaces
\end{itemize}
Surfaces of the first type are characterized by the condition $K_S\ne 0$ if $p =2$. Surfaces of the second and the third types  satisfy $K_S = 0$. They are distinguished by the action of the Frobenius endomorphism on the cohomology space $H^2(S,\calO_S) \cong \Bbbk$. It is trivial in the third case and it is not trivial in the second case.

Applying \eqref{r11}, we obtain the following.

\begin{corollary}\label{C7.2.5} The quotient group $\Aut_{\nt}(S)/\Aut_{\ct}(S)$ is a $2$-elementary abelian group.
\end{corollary}

\section{Half-fibers of genus one fibrations}
Recall that an Enriques surface always admits a fibration $f:S\to \bbP^1$  with general fiber $S_\eta$ an elliptic curve or a quasi-elliptic curve over the field $K$ of rational functions on $\bbP^1$ (i.e. a regular non-smooth irreducible curve of arithmetical genus one) (see \cite{CD}, Corollary 3.2.1). To treat both cases, we call such a fibration a \emph{genus one fibration}, specifying when needed whether it is an \emph{elliptic fibration} or a \emph{quasi-elliptic fibration}. 

A genus one fibration is defined by a base-point-free  pencil $|D|$ of divisors of arithmetic genus one satisfying $D^2 = 0$. The numerical class $[D]$ in $\Num(S)$ is always divisible by two, so $D = 2F$, where $[F]$ is a primitive isotropic vector in the lattice $\Num(S)$. There are two representatives $F,F'$ of $[F]$ if $p\ne 2$ or $S$ is classical Enriques surface in characteristic $2$. Otherwise, there is only one representative. We call these representatives \emph{half-fibers} of $|2F|$, of the pencil or of the corresponding fibration. 

Conversely, let $W_S^{\nod}$ be the group of isometries of $\Num(S)$ generated by reflections into the classes of smooth rational curves (\emph{$(-2)$-curves}, for short). Any primitive isotropic vector in $\Num(S)$ can be transformed by an element of  $W_S^{\nod}$ to the numerical class of a half-fiber. Hence, any nef divisor $F$ such that $[F]$ is a primitive isotropic vector in $\Num(S)$ defines a genus one pencil $|2F|$ and a corresponding genus one fibration $f:S\to \bbP^1$. An Enriques surface is called \emph{unnodal} if it does not contain $(-2)$-curves. In this case $W_S^{\nod} = \{1\}$ and there is a bijective correspondence between primitive isotropic vectors in $\Num(S)$ and genus one fibrations on $S$.

A general fiber of an elliptic (resp. quasi-elliptic) fibration is a smooth elliptic curve (resp. irreducible curve of arithmetic genus one with one ordinary cusp). We will use  the notation for singular fibers of elliptic fibrations (resp. reducible fibers of quasi-elliptic fibrations) 
$$\tilde{A}_0^*,\ \tilde{A}_{n-1},\ \tilde{D}_{n+4},\ \tilde{A}_0^{**},\ \tilde{A}_1^*,\ \tilde{A}_{2}^*,\ \tilde{E}_{8},\ \tilde{E}_{7}, \ \tilde{E}_{6}$$
from \cite{CD}. They correspond to Kodaira's notations $I_1,I_n,I_n^*,II,III,IV,II^*,III^*,IV^*$. Fibers of type $I_n$ are called of \emph{multiplicative type}, all others of \emph{additive type}. The notation indicates the relationship with Dynkin diagrams of affine root systems. In fact, the dual graph of irreducible components of a reducible fiber coincides with such a diagram.

We have the following (see \cite{CD}, Chapter 5. \S 7).

\begin{proposition}\label{half-fibers} Let $F$ be a half fiber of a genus one fibration on an Enriques surface.
\begin{itemize}
\item If $p\ne 2$ or $S$ is an ordinary Enriques surface in characteristic $2$, then $F$ is of multiplicative type or a smooth elliptic curve, which is ordinary if $p = 2$.
\item If $p = 2$ and $K_S\ne 0$, then $F$ is of additive type or a smooth ordinary elliptic curve. 
\item If $p = 2$ and $S$ is a supersingular Enriques surface, then $F$ is of additive type or a supersingular elliptic curve.
\end{itemize}
\end{proposition}  

A  $(-2)$-curve is called a \emph{special bisection} of a half-fiber $F$ or of the corresponding pencil $|2F|$, or of the corresponding genus one fibration, if it intersects $F$ with multiplicity $1$. 

A relatively minimal model of the Jacobian variety $J_\eta$ of the generic fiber $S_\eta$ of an elliptic fibration is a rational elliptic surface $j:J\to \bbP^1$. The group $J_\eta(\eta)$ is called the \emph{Mordell-Weil group} of the elliptic fibration. It is a finitely generated abelian group. It acts on $S_\eta$ by translation, and by the properties of a relative minimal model, the action extends to a biregular action on $S$. 

The type of a singular fiber $J_t$ of $j:J\to \bbP^1$ coincides with the type of the fiber $S_t$ (see \cite{CD}, Theorem 5.3.1 and \cite{Lorenzini}, Theorem 6.6). Similarly, if the fibration is quasi-elliptic, the Jacobian variety $J_\eta$ of its general fiber is a unipotent group scheme, a non-trivial inseparable form of the additive group scheme. Its Mordell-Weil group is a finite $p$-elementary abelian group. The theory of minimal models of surfaces allows us to construct a rational surface with a quasi-elliptic fibration whose generic fiber with the singular point deleted is isomorphic to $J_\eta$. 

An ordered  sequence $(f_1,\ldots,f_n)$ of isotropic vectors in $\Num(S)$ with $f_i\cdot f_j = 1-\delta_{ij}$ and $f_i\cdot h > 0$ for the class of an ample divisor $h$ can always be transformed by an element $w\in W_S^{\nod}$ to a sequence where $f_1+\cdots+f_n$ is the class of a nef divisor. A lift  $(F_1,\ldots,F_n)$ of such a sequence to $\NS(S)$ is called a \emph{$U_{[n]}$-sequence}. After reordering, we may assume that  $F_1$ is a half-fiber of a genus one  fibration and either $F_{i+1} = F_i + R$, where $R$ is a $(-2)$-curve with $R\cdot F_i = 1$ or $F_{i+1}$ is a half-fiber of a genus one fibration. A $U_{[n]}$-sequence is called \emph{c-degenerate}, if it contains exactly $c$ half-fibers.  If $c = n$, it is called \emph{non-degenerate}. We say that a $U_{[m]}$-sequence $A$ \emph{extends} a $U_{[n]}$-sequence $B$ if, after reordering, $A$ contains $B$. For a given Enriques surface $S$, the maximal length of a non-degenerate $U_{[n]}$-sequence is denoted by $\nd(S)$ and is called the \emph{non-degeneracy invariant} of $S$.

\begin{remark}\label{-2curves}
Note that, by definition, the $R_i$ that occur in a $U_{[n+1]}$-sequence of the form $(F_1,F_1+R_1,\hdots,F_1+\sum_{i=1}^{n} R_i)$ form a Dynkin diagram of type $A_n$ and the $R_i$ with $i \geq 2$ are contained in fibers of $|2F_1|$.
\end{remark}

For the following Proposition, see \cite{CD} Corollary 3.3.1.

\begin{proposition}\label{extend}
Let $n \leq 8$. Then, any $c$-degenerate $U_{[n]}$-sequence can be extended to a $c'$-degenerate $U_{[10]}$-sequence with $c' \geq c$.
\end{proposition}

It is a much more difficult  question whether a non-degenerate $U_{[n]}$-sequence can be extended to a non-degenerate $U_{[m]}$-sequence (see e.g. Section $5$). However, the following is known (see \cite{Cossec}, Theorem 3.5).

\begin{theorem}\label{diagrams} Suppose $p \neq 2$ or $S$ is an ordinary Enriques surface. Then, any half-fiber can be extended to a non-degenerate $U_{[3]}$-sequence. In particular, $\nd(S) \geq 3$.
\end{theorem} 


\begin{lemma}\label{nocommon} Let $F_1,F_2$ form a non-degenerate $U_{[2]}$-pair. Then, $F_1$ and $F_2$ do not have common irreducible components.
\end{lemma}

\begin{proof} We use that  a fiber $F_1$ is numerically $2$-connected, i.e. if we write $F_1$ as a sum of two proper effective divisors $F_1 = D_1+D_2$,  then  $D_1\cdot D_2 \ge 2$. To see this, we use that $D_1^2 < 0, D_2^2 < 0$ and $F_1^2 = F_1 \cdot D_1 = F_1 \cdot D_2 = 0$. Now, if $D_1$ is the maximal effective divisor with $D_1 \leq F_1$ and $D_1 \leq F_2$ and if we let $F_1 = D_1 + D_2$ and $F_2 = D_1+D_2'$ be decompositions into effective divisors, we have $D_2.D_2' \geq 0$. Therefore $1 = F_1 \cdot F_2= (D_1+D_2)\cdot F_2 = D_2 \cdot F_2 = D_2\cdot D_1+D_2\cdot D_2' \ge D_2\cdot D_1 $, where we use that $D_1.F_2 = 0$. Hence, $D_1 = 0$.
\end{proof}

Let $(F_1,F_2)$ be a non-degenerate $U_{[2]}$-sequence. Since $F_1\cdot F_2 = 1$, by the previous lemma, $F_1\cap F_2$ consists of one point. 

\begin{lemma}\label{L3.8} Let $(F_1,F_2,F_3)$ be a non-degenerate $U_{[3]}$-sequence. Suppose that $|F_2 + F_3 - F_1 + K_S| = \emptyset$.  Then, $F_1\cap F_2\cap F_3 =\emptyset$.
\end{lemma} 

\begin{proof} Consider the natural exact sequence coming from restriction of the sheaf $\calO_S(F_1-F_2)$ to $F_3$:
$$0\to \calO_S(F_1-F_2-F_3)\to \calO_S(F_1-F_2)\to \calO_{F_3}(F_1-F_2) \to 0.$$
We have  $(F_1-F_2-F_3)\cdot F_1 = -2$. Since $F_1$ is nef, the divisor class  $F_1-F_2-F_3$ is not effective. Thus, by Riemann-Roch and Serre's Duality,  $h^1(\calO_S(F_1-F_2-F_3)) = 0$ since $h^0(\calO_S(K_S+F_3+F_2-F_1)) = 0$ by assumption.  Now, $h^0(\calO_S(F_1-F_2)) = 0$, because $(F_1-F_2).F_1 = -1$ and $F_1$ is nef. Suppose $F_1\cap F_2\cap F_3 \ne \emptyset$, then $\calO_{F_3}(F_1-F_2)\cong \calO_{F_3}$ and $h^0(\calO_{F_3}(F_1-F_2)) = 1$. It remains to consider the exact sequence of cohomology and get a contradiction.
\end{proof}

\begin{remark} \label{remark1}
Note that for any $D \in |F_2+F_3-F_1 + K_S|$, we have $D^2 = -2$ and $D.F_2 = D.F_3 = 0$, so $D$ consists of $(-2)$-curves contained in fibers of $|2F_2|$ and $|2F_3|$.
\end{remark}

\section{Automorphisms of genus one curves}
Let us recall some known results about automorphism groups of elliptic curves over algebraically closed fields which we will use frequently. 
The proof of the following result can be found in \cite{Silverman}, III, \S 10 and Appendix A.

 \begin{proposition}\label{autoelliptic}  Let $E$ be an elliptic curve over an algebraically closed field with automorphism group $G$ and absolute invariant $j$. For $g \in G$, let $E^g$ be the set of fixed points of $g$.
 \begin{enumerate}
\item If $p\ne 2,3$
\begin{center}
\begin{tabular}{|>{\centering\arraybackslash}m{1.5cm}|>{\centering\arraybackslash}m{3cm}|>{\centering\arraybackslash}m{2.5cm}|>{\centering\arraybackslash}m{2.5cm}|}
\hline
$j$ & $G$ & $\ord(g)$ & \vspace{1mm} $|E^g|$\\ [1mm] \hline
$\neq 0,1$ & $\bbZ/2\bbZ$ & $2$ & $4$ \\ \hline
$1$ & $\bbZ/4\bbZ$ & 
$\begin{cases}
2 \\
4
\end{cases}$ &
$\begin{cases}
4 \\
2
\end{cases}$ \\ \hline
$0$ & $\bbZ/6\bbZ$ & 
$\begin{cases}
2\\
3 \\
6
\end{cases}$ &
$\begin{cases}
4 \\
3 \\
1 
\end{cases}$ \\ \hline
\end{tabular}
\end{center}
  \item If $p = 3$
\begin{center}
\begin{tabular}{|>{\centering\arraybackslash}m{1.5cm}|>{\centering\arraybackslash}m{3cm}|>{\centering\arraybackslash}m{2.5cm}|>{\centering\arraybackslash}m{2.5cm}|}
\hline
$j$ & $G$ & $\ord(g)$ & \vspace{1mm} $|E^g|$\\ [1mm] \hline
$\neq 0$ & $\bbZ/2\bbZ$ & $2$ & $4$ \\ \hline
$0$ & $\bbZ/3\bbZ \rtimes \bbZ/4\bbZ$ & 
$\begin{cases}
2 \\
3 \\
4
\end{cases}$ &
$\begin{cases}
4 \\
1 \\
2
\end{cases}$ \\ \hline
\end{tabular}
\end{center}
 
  \item If $p = 2$
\begin{center}
\begin{tabular}{|>{\centering\arraybackslash}m{1.5cm}|>{\centering\arraybackslash}m{3cm}|>{\centering\arraybackslash}m{2.5cm}|>{\centering\arraybackslash}m{2.5cm}|}
\hline
$j$ & $G$ & $\ord(g)$ & \vspace{1mm} $|E^g|$\\ [1mm] \hline
$\neq 0$ & $\bbZ/2\bbZ$ & $2$ & $2$ \\ \hline
$0$ & $Q_8 \rtimes \bbZ/3\bbZ$ & 
$\begin{cases}
2,4 \\
3
\end{cases}$ &
$\begin{cases}
1 \\
3
\end{cases}$ \\ \hline
\end{tabular}
\end{center}

\end{enumerate}
\end{proposition} 

\section{Bielliptic maps and bielliptic involutions}\label{S:4} Let $(F_1,F_2)$ be a non-degenerate $U_{[2]}$-pair of half-fibers. The linear system $|2F_1+2F_2|$ defines a morphism of degree 2 from $S$ to a surface $\sfD$ of degree $4$ in $\bbP^4$ (it is called a superelliptic map in \cite{CD},  renamed as a bielliptic map in \cite{CDL}). The surface $\sfD$ is an anti-canonical model of a unique (up to isomorphism) weak del Pezzo surface of degree 4 obtained by blowing up 5 points $p_1,\ldots,p_5$ in the projective plane $\bbP^2$.

If $K_S\ne 0$, the point $p_3$ is infinitely near to $p_2$ and $p_5$ is infinitely near to $p_4$. The points $p_1,p_2,p_3$ and $p_1,p_4,p_5$ lie on lines $\ell_1$ and $\ell_1'$. The proper inverse transform of the pencil of lines through $p_1$ and the pencil of conics through $p_2,p_3,p_4,p_5$ on $\bbP^2$ are pencils of conics on $\sfD$. The exceptional curves over $p_3$ and $p_5$ (resp. the line passing through $p_2,p_4$ and the exceptional curve over $p_1$) are the four lines $L_1,L_1'$ (resp. $L_2,L_2'$) on $\sfD$. The proper inverse transforms  of the two pencils of conics on $\sfD$ are the  genus one pencils $|2F_1|$ and $|2F_2|$ of $S$. The half-fibers $F_1,F_1'$ (resp. $F_2,F_2'$) are the proper inverse transforms of the lines $L_1,L_1'$ (resp. $L_2,L_2'$). One can choose projective coordinates in $\bbP^4$ so that $\sfD$ is given by equations
\beq\label{eqd1}
x_0^2+x_1x_2 = x_0^2+x_3x_4 = 0.
\eeq
The pencils of conics that give rise to the pencils $|2F_1|$ and $|2F_2|$ are cut out by the linear pencils of planes
\beq\label{conics}
ax_2 +bx_3 = ax_4+bx_1 = 0, \quad  ax_2 +bx_4  =  ax_3+bx_1 =  0.
\eeq
 The lines are given by equations $x_0 = x_i=x_j = 0, i\in \{1,2\},j\in \{3,4\}$. They correspond to the parameters $(a:b) = (1:0)$ and $(0:1)$. 

If $K_S = 0$ and $S$ is ordinary (resp. supersingular), the surface $\sfD$ has a unique singular point, which is a rational double of type $D_4^{1}$ (resp. $D_4^{0}$) in the notation of Artin \cite{Artin}. The surface is again an anti-canonical model of a unique (up to isomorphism) weak del Pezzo surface of degree 4, which is the blow-up of 5 points $p_1,\ldots,p_5$ in $\bbP^2$, where $p_5$ is infinitely near to $p_4$, $p_4$ is infinitely near to $p_3$ and $p_3$ is infinitely near to $p_2$. The points $p_1,p_2$ and $p_3$ lie on a line $l$, but $p_4$ and $p_5$ are not on $l$. The surface $\sfD$, obtained by contracting the proper inverse transform of $l$ and the exceptional curves over $p_2,p_3$ and $p_4$, contains only two lines, which are the exceptional curves over $p_1$ and $p_5$. Their 
proper inverse transforms on $S$ are the half-fibers of the genus one fibrations $|2F_1|$ and $|2F_2|$. The fibrations themselves are defined by the pencils of conics on $\sfD$ obtained from the pencil of lines through $p_1$ and the pencil of conics through the points $p_2,p_3,p_4,p_5$. The surface $\sfD$ is isomorphic to a surface given by equations
\beq\label{eqd2}
x_0^2+x_1x_2=x_1x_3+x_4(ex_0+x_2+x_4) = 0,
\eeq
where $e = 1$ if $S$ is ordinary, and $e=0$ if $S$ is supersingular.  
The pencils of concis that give rise to our pencils are given by the equations
\beq\label{conics2}
a x_3 +b (ex_0+x_2+x_4) = a x_4+ b x_1 = 0, \quad  a (ex_0+x_2+x_4)+ b x_1 = a x_3 + b x_4 =  0.
\eeq

If the map $\phi$ is separable, the birational automorphism of $S$ defined by the degree two separable extension of the fields of rational functions $\Bbbk(S)/\phi^*\Bbbk(\sfD)$ extends to a biregular automorphism of $S$ which we call a \emph{bielliptic involution} of $S$.  

The group of automorphisms of the surface $\sfD$ is a subgroup of projective transformations of $\bbP^4$ that leaves the surface $\sfD$ invariant. The following proposition describes the group of automorphisms of the quartic surface $\sfD$ and we leave the computations to the reader.\footnote{The computation of these groups in the cases of surfaces $\sfD_2,\sfD_3$ in \cite{CD} is erroneous. The correct computation can be found in \cite{CDL}.}

\begin{proposition}\label{automorphisms} Let $\sfD_1,\sfD_2,\sfD_3$ be the image of a bielliptic map defined by the linear system $|2F_1+2F_2|$, where $K_S\ne 0$, $S$ is ordinary, or $S$ is supersingular, respectively. Then 
\begin{itemize}
\item $\Aut(\sfD_1)\cong \bbG_m^2\rtimes D_8$;
\item $\Aut(\sfD_2)\cong \bbG_a^2\rtimes \bbZ/2\bbZ$;
\item $\Aut(\sfD_3)\cong (\bbG_a^2\rtimes \bbG_m)\rtimes \bbZ/2\bbZ$.
\end{itemize}
Here, $\bbG_m$ (resp. $\bbG_a$) denotes the multiplicative (resp. additive) one-dimensional algebraic group over $\Bbbk$ and $D_8$ denotes the dihedral group of order $8$.
\end{proposition}

\begin{remark}\label{explicit}
Note that the connected component $\Aut(\sfD)^0$ of $\Aut(\sfD)$ is the group of automorphisms preserving each line on $\sfD$.
Using equations (\ref{eqd1}) and (\ref{eqd2}), we can write the action of $\Aut(\sfD)^0$ explicitly as follows, with $\lambda,\mu \in \bbG_m$ and $\alpha,\beta \in \bbG_a$:

\begin{itemize}
\item Action of $\Aut(\sfD_1)^0:$ 
\begin{equation*}
(x_0:x_1:x_2:x_3:x_4) \mapsto (x_0:\lambda x_1: \lambda^{-1}x_2 : \mu x_3 : \mu^{-1} x_4)
\end{equation*}
\item Action of $\Aut(\sfD_2)^0:$
{\small \begin{equation*} (x_0:x_1:x_2:x_3:x_4) \mapsto (x_0 + \alpha x_1 : x_1 : \alpha^2 x_1 + x_2 : \beta x_0 + (\alpha \beta+\alpha^2 \beta + \beta^2)x_1 + \beta x_2 + x_3 + (\alpha + \alpha^2)x_4 : \beta x_1 + x_4)
\end{equation*}}
\item Action of $\Aut(\sfD_3)^0:$
 \begin{align*}
(x_0:x_1:x_2:x_3:x_4) &\mapsto (x_0+\alpha x_1:x_1:\alpha^2 x_1+x_2:(\alpha^2\beta+\beta^2) x_1+\beta x_2+x_3+\alpha^2x_4,\beta x_1+x_4) \\
(x_0:x_1:x_2:x_3:x_4) &\mapsto (x_0:\lambda^{-1}x_1:\lambda x_2:\lambda^3x_3,\lambda x_4)
\end{align*}
\end{itemize}

Moreover, we can compute the group of automorphism fixing the pencils given by equations (\ref{conics}) (resp. (\ref{conics2}))  on $\sfD$. They are obtained by setting $\lambda = \mu \in \{1,-1\}$ (resp. $\alpha \in \{0,1\},\beta = 0$, resp. $\alpha = \beta = 0$, $\lambda = 1$).

\end{remark}

The known information about the automorphism group of the surfaces $\sfD$ allows us to give a criterion for an automorphism to be a bielliptic involution.

\begin{corollary} \label{criterion}
Let $(F_1,F_2)$ be a non-degenerate $U_{[2]}$-sequence and let $g$ be a non-trivial automorphism of $S$. Assume that $g$ preserves $F_1$, $F_2$ and a $(-2)$-curve $E$ with $E.F_1 = E.F_2 = 0$, which is not a component of one of the half-fibers $F_1,F_2,F_1',F_2'$. If $S$ is supersingular, assume additionally that $g$ has order $2^n$. Then, $g$ is the bielliptic involution associated to the linear system $|2F_1 + 2F_2|$.
\end{corollary}

\begin{proof}
Let $\phi: S \to \sfD$ be a bielliptic map defined by the linear system $|2F_1 + 2F_2|$. Since $g$ leaves  $|2F_1 + 2F_2|$ invariant, it descends to an automorphism of $\bbP^4 = |2F_1 + 2F_2|^*$ that leaves $\sfD$ invariant. Moreover, the induced automorphism preserves the lines on $\sfD$ by assumption. Recall that $E.F_1 = E.F_2 = 0$, hence $\phi(E)$ is a point $P$. Since $E$ is not a component of one of the half-fibers, $P$ does not lie on any of the lines of $\sfD$. If $\sfD = \sfD_1$, this means that $P$ is not on the hypersurface $x_0 = 0$ and if $\sfD \in \{\sfD_2,\sfD_3\}$, it means that $P$ is not on the hypersurface $x_1 = 0$. 

If $\sfD = \sfD_1$, the $x_0$ coordinate $x_0(P)$ of $P$ is non-zero, hence so are all $x_i(P)$ by Equation (\ref{eqd1}). By Remark \ref{explicit}, there is no automorphism of $\sfD_1$ fixing $P$ and preserving the lines except the identity. Therefore, $g$ coincides with the covering involution of $\phi$.

If $\sfD  \in \{\sfD_2,\sfD_3\}$, we have $x_1(P) \neq 0$. Again, by Remark \ref{explicit}, there is no automorphism of $\sfD_2$ fixing $P$ and preserving the lines except the identity. For $\sfD_3$, we use the additional assumption to exclude the case that $g$ acts on $\sfD_3$ via $\bbG_m$.
\end{proof}

\begin{remark}
In fact, the failure of this criterion without the additional assumption in the supersingular case leads to the existence of cohomologically trivial automorphisms of odd order (see Section $7$).
\end{remark}

\begin{lemma}\label{numtrivialbiell} Let  $\tau$ be the bielliptic involution associated to a linear system $|2F_1+2F_2|$. Suppose $\tau$ is numerically trivial. Then, $\Num(S)_\bbQ$ is spanned by the numerical classes $[F_1],[F_2]$ and eight  
smooth rational curves that are contained in fibers of both $|2F_1|$ and $|2F_2|$. 
\end{lemma}

\begin{proof} We have a finite degree 2 cover $S' = S-E\to D' = D-P$, where $E$ is spanned by (-2)-curves blown down 
to a finite set of points $P$ on $\sfD$. We have $\Pic(\sfD')_{\bbQ} = \Pic(\sfD)_{\bbQ}$ and $\Pic(S')_{\bbQ}^g$ (the invariant part) = $f^*(\Pic(\sfD')_{\bbQ}$ is spanned by the restriction of $F_1, F_2$ to $S'$.  Since $\Pic(S)$ is spanned by $\Pic(S')$ and the classes of components of $E$, we  can write any invariant divisor  class as a linear combination of $[F_1],[F_2]$ and invariant components of $E$. In our case all divisors classes are invariant. Since $\dim( \Pic(S)_\bbQ )- \dim (\langle F_1,F_2 \rangle_\bbQ) = 8$, $E$ consists of eight $(-2)$-curves.
\end{proof}

Denote the number of irreducible components of a fiber $D$ of $|2F|$ by $m_{D}$. Since $\rank(\Pic(S)) = 10$, we have $\sum_{D \in |2F|} (m_{D}-1) \leq 8$, and, by the Shioda-Tate formula, the Jacobian of $|2F|$ has finite Mordell-Weil group if and only if equality holds. In the latter case, $|2F|$ is called \emph{extremal}.

\begin{corollary}\label{extremal}
Let $(F_1,F_2)$ be a $U_{[2]}$-pair of half-fibers such that the bielliptic involution $\tau$ associated to $|2F_1+2F_2|$ is numerically trivial. Then, $|2F_1|$ and $|2F_2|$ are extremal.

Moreover, the following hold:
\begin{enumerate}
\item  For every fiber $D$ of $|2F_1|$, all but one component $C$ of $D$ is contained in fibers  of $|2F_2|$.
\item  $C$ has multiplicity at most $2$.
\item  Neither $|2F_1|$ nor $|2F_2|$ have a multiplicative fiber with more than two components.
\end{enumerate} 
\end{corollary}

\begin{proof}
By the previous lemma, there are eight $(-2)$-curves contained in fibers of both $|2F_1|$ and $|2F_2|$. Since a fiber of $|2F_1|$ cannot contain a full fiber of $|2F_2|$, this implies $8 \leq \sum_{D \in |2F_i|} (m_{D}-1) \leq 8$ for $i \in \{1,2\}$. Hence, $|2F_1|$ is extremal and so is $|2F_2|$. Moreover, if, for some fiber $D$ of $|2F_1|$, two components of $D$ are not contained in fibers of $|2F_2|$, then, by the same formula, $|2F_1|$ and $|2F_2|$ share less than eight $(-2)$-curves. This contradicts Lemma \ref{numtrivialbiell}.

For $(2)$, note that the remaining component $C$ of multiplicity $m$ in $D$ satisfies $2 = D.F_2 = mC.F_2$.  Since $C.F_2 > 0$, this yields $(2)$.

As for $(3)$, assume that $D$ is multiplicative with more than $2$ components. Note that $C$ meets distinct points on distinct components of $D$. The connected divisor $D' = D-C$ satisfies $D'.(2F_1 + 2F_2) = 0$, hence it is contained in the exceptional locus of the bielliptic map $\phi$. Since $\tau$ preserves the components of $D'$, $\phi(C)$ is an irreducible curve with a node. But $C$ is contained in the pencil of conics induced by $|2F_1|$. This is a contradiction.
\end{proof}

\section{Extra-special Enriques surfaces}
When trying to apply our knowledge of bielliptic maps and the geometry of the surfaces $\sfD$ to the study of automorphisms of an Enriques surface $S$, we first have to make sure that $S$ admits such a bielliptic map. This is guaranteed by $\nd(S) > 1$. On the other hand, even if $S$ admits a bielliptic map, it might be the case that the $(-2)$-curves on $S$ are in a very special position relative to this map and we will need to be able to choose a different bielliptic map with better properties, e.g. to be able to apply Corollary \ref{criterion}. It turns out that, with our method, surfaces with $\nd(S) \geq 3$ can be treated in a very uniform way whereas \emph{extra-special} Enriques surfaces, i.e. those with $\nd(S) \leq 2$, which only exist in characteristic $2$ by Theorem \ref{diagrams}, show special behaviour (see Table \ref{exceptnun1}). Fortunately, these extra-special surfaces can be classified and we recall the classification in this section.


It is claimed in \cite{CD}, Theorem 3.5.1 that Theorem \ref{diagrams} is true in any characteristic unless the surface is extra-special with finitely many $(-2)$-curves with the dual graph defined by one of the diagrams from the following  Table \ref{extraspecialtable}. The surfaces of type $\tilde{E}_8$, $\tilde{E}_7^1$ and $\tilde{D}_8$ are  called $E_8$, $E_7$ and $D_8$-extra-special, respectively.  However, the surface of type $\tilde{E}_7^2$ was erroneously asserted to have $\nd(S) = 2$, although, in fact, it is not extra-special and has $\nd(S) = 3$, as can be checked by computing the intersection numbers of the fibers of the three genus $1$ fibrations of this surface (see \cite{DL}).\footnote{So far, this is the only known example of an Enriques surface with $\nd(S) = 3$.} Also, the proof of the claim is too long (occupies more than 30 pages of case-by-case arguments) and it is difficult to verify that the authors have not omitted some possible cases. We refer the reader to \cite{DL} for a different proof due to the second author of the classification of extra-special surfaces and collect the results we need in the context of numerically trivial automorphisms in this section.

\begin{table}[h!]
\begin{center}
\begin{tabular}{|>{\centering\arraybackslash}m{1.5cm}|>{\centering\arraybackslash}m{7cm}|>{\centering\arraybackslash}m{2.5cm}|}
\hline
Type& Configuration\\ \hline
$\tilde{E}_8$& 
\resizebox{!}{1cm}{
\xy
(-10,25)*{};
@={(-10,10),(0,10),(10,10),(20,10),(30,10),(40,10),(50,10),(60,10),(70,10),(10,20)}@@{*{\bullet}};
(-10,10)*{};(70,10)*{}**\dir{-};
(10,10)*{};(10,20)*{}**\dir{-};
\endxy
}\\ \hline
$\tilde{E}_7^1$&
\resizebox{!}{1cm}{
\xy
(0,25)*{};
@={(80,10),(90,10),(0,10),(10,10),(20,10),(30,10),(40,10),(50,10),(60,10),(70,10),(30,20)}@@{*{\bullet}};
(0,10)*{};(80,10)*{}**\dir{-};
(90,10)*{};(80,10)*{}**\dir{=};
(30,10)*{};(30,20)*{}**\dir{-};
\endxy
} 
\\ \hline
$\tilde{E}_7^2$&
\resizebox{!}{1cm}{
\xy
(0,25)*{};
@={(80,10),(90,10),(0,10),(10,10),(20,10),(30,10),(40,10),(50,10),(60,10),(70,10),(30,20)}@@{*{\bullet}};
(0,10)*{};(80,10)*{}**\dir{-};
(90,10)*{};(80,10)*{}**\dir{=};
(30,10)*{};(30,20)*{}**\dir{-};
(70,10)*{};(90,10)*{}**\crv{(80,20)};
\endxy
} 
\\ \hline
$\tilde{D}_8$&
\resizebox{!}{1cm}{
\xy
(-10,25)*{};
@={(-10,10),(0,10),(10,10),(20,10),(30,10),(40,10),(50,10),(10,20),(60,10),(50,20)}@@{*{\bullet}};
(-10,10)*{};(60,10)*{}**\dir{-};
(10,10)*{};(10,20)*{}**\dir{-};
(50,10)*{};(50,20)*{}**\dir{-};
\endxy
} 
\\ \hline
\end{tabular}
\end{center}
\caption{$E_8$,$E_7$ and $D_8$-extra-special surfaces and the $\tilde{E}_7^2$ surface}
\label{extraspecialtable}
\end{table}

%
%
%
%


\begin{theorem}\label{diagrams2}
Assume that $S$ is not $E_8$-extra-special. Then, any half-fiber can be extended to a non-degenerate $U_{[2]}$-sequence. In particular, $\nd(S) \geq 2$.
\end{theorem}

\begin{theorem} \label{extraspecialclassification}
Assume that $S$ is not $E_8$,$E_7$ or $D_8$-extra-special. Then, $\nd(S) \geq 3$.
\end{theorem}
%

\begin{remark}\label{extraspecialauts}
In \cite{KKM}, the cohomologically trivial and numerically trivial automorphism groups of extra-special surfaces have been calculated. For their examples, the groups are given in Table \ref{exceptnun1}.

\begin{table}[h!]
\begin{center}
\begin{tabular}{|>{\centering\arraybackslash}m{4.5cm}|>{\centering\arraybackslash}m{1.5cm}|>{\centering\arraybackslash}m{1.5cm}|>{\centering\arraybackslash}m{1.5cm}|}
\hline
Type&$\Aut_{\ct}(S)$& \vspace{0.5mm} $\Aut_{\nt}(S)$\\ [0.5mm] \hline
classical $\tilde{E}_8$ & $\{1\}$&  \vspace{0.5mm} $\{1\}$\\ [0.5mm] \hline
supersingular $\tilde{E}_8$  &  $\bbZ/11\bbZ$ &  \vspace{0.5mm}  $\bbZ/11\bbZ$ \\  [0.5mm]  \hline
classical $\tilde{D}_8$&$\bbZ/2\bbZ$&  \vspace{0.5mm} $\bbZ/2\bbZ$\\ [0.5mm] \hline
supersingular $\tilde{D}_8$&$Q_8$&  \vspace{0.5mm} $Q_8$ \\ [0.5mm] \hline
classical $\tilde{E}_7^1$&$\{1\}$&  \vspace{0.5mm} $\bbZ/2\bbZ$\\ [0.5mm] \hline
\end{tabular}
\end{center}
\caption{Numerically trivial automorphisms of extra-special surfaces}
\label{exceptnun1}
\end{table}

However, it is not known whether there are more surfaces of these types than the ones given in \cite{KKM}. Note that the calculation of these groups in the case where $S$ is classical of type $\tilde{D}_8$ or $\tilde{E}_7^1$ only depends on the dual graph of $(-2)$-curves.
\end{remark}
\section{Cohomologically trivial automorphisms}
Now that we have treated the necessary background material, we can proceed to the heart of our paper. In this section, we prove our main results on cohomologically trivial automorphisms.

\subsection{Cohomologically trivial automorphisms of even order}

\begin{theorem}\label{cohmain}
Let $S$ be an Enriques surface which is not extra-special.
\begin{enumerate}
\item If $S$ is classical or ordinary, then $|\Aut_{\ct}(S)| \leq 2$. If $S$ is also unnodal, then $\Aut_{\ct} = \{1\}$.
\item If $S$ is supersingular, then the statements of $(1)$ hold for the $2$-Sylow subgroup $G$ of $\Aut_{\ct}(S)$.
\end{enumerate} 
Moreover, if a non-trivial $g \in \Aut_{\ct}$ (resp. $G$) exists, then $g$ is a bielliptic involution.
\end{theorem}

\begin{proof}
Let $g \in \Aut_{\ct}(S)$ and assume that $g$ has order $2^n$ if $S$ is supersingular. Note that, by definition, $g$ preserves all half-fibers on $S$. We will show that there is a $U_{[2]}$-pair such that $g$ satisfies the conditions of Corollary \ref{criterion}. Note that $g$ preserves all half-fibers and $(-2)$-curves, since it is cohomologically trivial, so it suffices to find a $(-2)$-curve, which is contained in two simple fibers of genus one fibrations forming a $U_{[2]}$-pair.

Take a $c$-degenerate $U_{[10]}$-sequence on $S$ with $c$ maximal, i.e. $c = \nd(S)$. Since we assumed that $S$ is not extra-special, we have $c = \nd(S) \geq 3$ by definition. If $c \leq 9$, then there is a $(-2)$-curve $R$ in this sequence such that $F_i.R = 0$ for at least $3$ half-fibers $F_1,F_2$ and $F_3$ in the sequence (see Remark \ref{-2curves}). Now, Lemma \ref{nocommon} shows that $R$ is contained in at most one of the half-fibers $F_1,F_2$ and $F_3$. Therefore, without loss of generality, $R$ is contained in a simple fiber of the two pencils $|2F_1|$ and $|2F_2|$. Since $g$ preserves $R$, $g$ is the bielliptic involution associated to $|2F_1 + 2F_2|$ by Corollary \ref{criterion}. In particular, $g$ is unique.

If $c = 10$, assume that one of the half-fibers, say $F_1$, is reducible. Then, by Lemma \ref{nocommon}, for every $F_i$ in the sequence, all but one component of $F_1$ is contained in simple fibers of $|2F_i|$. Hence, we find some component $R$ with $R.F_i = 0$ for at least $3$ half-fibers and the same argument as before applies. 

If $|F_i+F_j-F_k| \neq \emptyset$ or $|F_i + F_j - F_k + K_S| \neq \emptyset$ for some half-fibers $F_i,F_j,F_k$ occurring in the sequence, by Remark \ref{remark1}, there is an effective divisor $D$ with $D.F_i = D.F_j = 0$ and $D^2 = -2$. Since $F_i$ and $F_j$ can be assumed to be irreducible, $D$ contains a $(-2)$-curve which is contained in a simple fiber of both $|2F_i|$ and $|2F_j|$. Again, Corollary \ref{criterion} applies.

Therefore, we can assume that all half-fibers are irreducible and $F_i \cap F_j \cap F_k = \emptyset$ by Lemma \ref{L3.8}, so $g$ has at least $9$ distinct smooth fixed points on every half-fiber. Since a non-trivial automorphism of a nodal or cuspidal plane cubic has at most $2$ smooth fixed points and a non-trivial automorphism of an elliptic curve has at most $4$ fixed points by Proposition \ref{autoelliptic}, $g$ fixes all $F_i$ pointwise, hence it is trivial, as can be seen by applying the same Proposition to a general fiber of, say, $|2F_1|$. If $S$ is unnodal, then all half-fibers of genus $1$ fibrations on $S$ are irreducible and no three half-fibers in a $U_{[n]}$-sequence intersect in a common point by Remark \ref{remark1}, so the argument of this paragraph shows that $\Aut_{ct}(S)$ (resp. $G$) is trivial if $S$ is unnodal.
\end{proof}

In the case of classical Enriques surface in characteristic $2$, we can say more, using the classification of Enriques surfaces with finite automorphism group.

\begin{corollary}\label{mainclassical}
Let $S$ be a classical Enriques surface in characteristic $2$ which is not $E_8$-extra-special. Then, $\Aut_{\ct}(S) \cong \bbZ/2\bbZ$ if and only if $S$ is $D_8$-extra-special.
\end{corollary}

\begin{proof}
Let $F_1$ be a half-fiber on $S$. By Theorem \ref{diagrams2}, we can extend $F_1$ to a non-degenerate $U_{[2]}$-sequence. Assume that there exists a non-trivial $g \in \Aut_{\ct}(S)$. Then, $g$ acts on $\sfD_1$ via its action on $|2F_1 + 2F_2|^*$. By Proposition \ref{automorphisms}, $g$ acts via $\bbG_m^2$ on $\sfD_1$. But $g$ has order $2$ by Theorem \ref{cohmain}, hence it acts trivially on $\sfD_1$. Therefore, $g$ is the covering involution of the bielliptic map and by Corollary \ref{extremal}, $|2F_1|$ is extremal. Therefore, every genus one fibration on $S$ is extremal. In particular, by \cite{KKM} Section $12$, $S$ has finite automorphism group. The groups $\Aut_{\ct}(S)$ of these surfaces have been calculated in \cite{KKM} and the only surfaces for which the calculation of the groups depends on the specific example given in \cite{KKM} are the ones of type $\tilde{E}_8$ and $\tilde{D}_4 + \tilde{D}_4$ (see Remark \ref{extraspecialauts} and Remark \ref{D4D4}). In the latter case, there is a $U_{[2]}$-pair of fibrations with simple $\tilde{D}_8$ fibers, which share only $7$ components. By Corollary \ref{extremal}, the corresponding bielliptic involution is not cohomologically trivial. Therefore, the calculation of the groups in \cite{KKM} shows that the $D_8$-extra-special surface is the only classical Enriques surface which is not $E_8$-extra-special and has a non-trivial cohomologically trivial automorphism.
\end{proof}

\begin{remark} Using Theorem \ref{cohmain} and Corollary \ref{extremal} may lead to an explicit classification of Enriques surfaces $S$ with non-trivial $\Aut_{\nt}(S)$. For example, in characteristic $p\ne 2$, one can show that the surface must contain $(-2)$-curves with one of the following dual graphs:

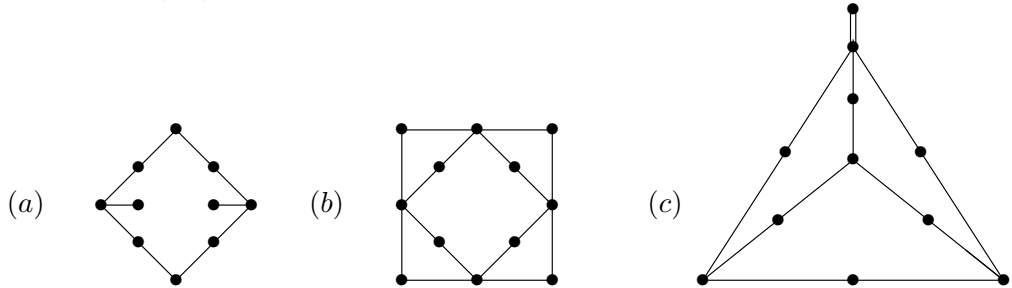
\centerline{
\xy
(0,-20)*{};(0,20)*{};
@={(0,0),(20,0),(10,10),(5,5),(5,-5),(10,-10),(15,5),(15,-5),(5,0),(15,0)}@@{*{\bullet}};
(0,0)*{};(10,10)*{}**\dir{-};(0,0)*{};(5,0)*{}**\dir{-};(20,0)*{};(15,0)*{}**\dir{-};
(0,0)*{};(10,-10)*{}**\dir{-};
(10,-10)*{};(20,0)*{}**\dir{-};
(10,10)*{};(20,0)*{}**\dir{-};(-10,0)*{(a)};
@={(40,0),(60,0),(40,10),(40,-10),(60,10),(60,-10),(50,10),(50,-10), (45,5),(45,-5), (55,5),(55,-5)}@@{*{\bullet}};
(40,10)*{};(40,-10)*{}**\dir{-};(40,10)*{};(60,10)*{}**\dir{-};(40,-10)*{};(60,-10)*{}**\dir{-};
(40,0)*{};(50,-10)*{}**\dir{-};(60,10)*{};(60,-10)*{}**\dir{-};
(40,0)*{};(50,10)*{}**\dir{-};
(50,10)*{};(60,0)*{}**\dir{-};(50,-10)*{};(60,0)*{}**\dir{-};
(30,0)*{(b)};
(0,18)*{};
@={(80,-10),(120,-10),(100,26),(100,6),(90,-2),(100,21),(100,14),(110,-2),(100,-10),(91,7),(109,7)}@@{*{\bullet}};
(85,-6)*{};(100,6)*{}**\dir{-};(100,22)*{};(100,6)*{}**\dir{-};
(80,-10)*{};(120,-10)*{}**\dir{-};(80,-10)*{};(100,21)*{}**\dir{-};(120,-10)*{};(100,21)*{}**\dir{-};
(80,-10)*{};(85,-6)*{}**\dir{-};(120,-10)*{};(115,-6)*{}**\dir{-};(100,26)*{};(100,21)*{}**\dir2{-};
(120,-10)*{};(100,6)*{}**\dir{-};
(75,0)*{(c)};
\endxy}
In the case $\Bbbk= \bbC$ this is an assertion from \cite[Theorem (1.7)]{Kondo}. We hope to address this problem in another paper.

\end{remark}

\subsection{Cohomologically trivial automorphisms of odd order}
Before we start with the treatment of cohomologically trivial automorphisms of odd order of supersingular Enriques surfaces, let us collect the known examples. These surfaces have finite automorphism groups and a detailed study can be found in \cite{KKM}. In Table \ref{table3}, we give the group of cohomologically trivial automorphisms of these examples. Again, it is not known whether there are more examples of these surfaces than the ones given in \cite{KKM}. 
\begin{table}[h!] 
\begin{center}
\begin{tabular}{|>{\centering\arraybackslash}m{1.5cm}|>{\centering\arraybackslash}m{2.5cm}|}
\hline
Type & \vspace{0.5mm} $\Aut_{\ct}(S)$\\ [0.5mm] \hline
$\tilde{E}_8$
&  \vspace{0.5mm}  $\bbZ/11\bbZ$\\  [0.5mm]  \hline
$\tilde{E}_7^2$
& \vspace{0.5mm}  $\bbZ/7\bbZ$ or $\{1\}$\\  [0.5mm]  \hline
$\tilde{E}_6$
& \vspace{0.5mm}  $\bbZ/5\bbZ$\\ [0.5mm]  \hline
\end{tabular}
\end{center}
\caption{Examples of cohomologically trivial automorphisms of odd order}
\label{table3}
\end{table}

\begin{remark}\label{E6}
The dual graph of $(-2)$-curves on a surface of type $\tilde{E}_6$ is as follows:

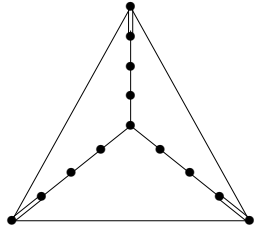
\centerline{
\resizebox{!}{3cm}{
\xy 
(0,38)*{};
@={(0,0),(40,0),(20,36),(20,16),(5,4),(10,8),(15,12),(20,21),(20,26),(20,31),(35,4),(30,8),(25,12)}@@{*{\bullet}};
(5,4)*{};(20,16)*{}**\dir{-};(35,4)*{};(20,16)*{}**\dir{-};(20,32)*{};(20,16)*{}**\dir{-};
(0,0)*{};(40,0)*{}**\dir{-};(0,0)*{};(20,36)*{}**\dir{-};(40,0)*{};(20,36)*{}**\dir{-};
(0,0)*{};(5,4)*{}**\dir2{-};(40,0)*{};(35,4)*{}**\dir2{-};(20,36)*{};(20,31)*{}**\dir2{-};
\endxy
}}

\noindent This surface is also called an \emph{exceptional Enriques surface of type $\tilde{E}_6$}. For more details, see \cite{EkeShep}, \cite{CDL}, and \cite{KKM}.
\end{remark}

\begin{lemma} \label{trivial2}
Let $S$ be a supersingular Enriques surface which is not $E_8$-extra-special and let $G \subseteq \Aut_{\ct}(S)$ be a non-trivial subgroup of odd order. Then, $G$ is cyclic and acts non-trivially on the base of every genus one fibration of $S$.
\end{lemma}

\begin{proof}
Take any half-fiber $F_1$ and extend it to a non-degenerate $U_{[2]}$-sequence $(F_1,F_2)$ on $S$. Since $G$ has odd order, it acts on $\sfD_3$ via a finite subgroup of $\bbG_m$, hence $G$ is cyclic.
By Remark \ref{explicit}, a generator $g$ of $G$ acts on the image $\sfD_3$ of the bielliptic map as
\begin{equation*}
(x_0:x_1:x_2:x_3:x_4) \mapsto (x_0:\lambda^{-1}x_1:\lambda x_2:\lambda^3x_3,\lambda x_4).
\end{equation*}
Such an automorphism acts non-trivially on the pencils of conics given by Equation $(\ref{conics2})$, hence $g$ acts non-trivially on $|2F_1|$.
\end{proof}

\begin{lemma}\label{lefschetz} Let $F$ be a fiber of a genus one fibration and let $g$ be a tame automorphism of finite order that fixes the irreducible components of $F$. Then,  the Lefschetz fixed-point formula 
$$e(F^g) = \sum_{i=0}^2 (-1)^i \textrm{Tr}(g^*|H_{\et}^i(F,\bbQ_l)).$$
holds for $F$.  If $F$ is reducible and not of type  $\tilde{A}_1$, then $e(F^g) = e(F)$. If $F$ is of type $\tilde{A}_1$, then  
$e(F^g) =e(F) = 2$ or $e(F^g) = 4$. The latter case can only occur if $g$ has even order.
\end{lemma}

\begin{proof} In the case the order is equal to $2$, this is proven in \cite{DolgachevNum} by a case-by-case direct verification. The proof uses only the fact that a tame non-trivial automorphism of finite order of $\bbP^1$ has two fixed points. Also note that the verification in case $F$ is of type $\tilde{A}_1$ and $g$ interchanges the two singular points of $F$ was missed, but it still agrees with the Lefschetz formula. 
\end{proof}

\begin{proposition}\label{possiblecases}
Let $g \in \Aut_{\ct}(S)$ be a non-trivial automorphism of odd order. Then, every genus one pencil $|D|$ of $S$ has one of the following combinations of singular fibers
\beq
\tilde{D}_4+\tilde{D}_4,\ \tilde{D}_8+\tilde{A}_0^{**},\ \tilde{E}_6+\tilde{A}_2^*,\  \tilde{E}_7+\tilde{A}_1^{*},\  \tilde{E}_8+\tilde{A}_0^{**},\  \tilde{A}_8 + \tilde{A}_0 + \tilde{A}_0 + \tilde{A}_0,\ \tilde{D}_7,\ \tilde{E}_7
\eeq

The last three configurations can only occur if $g$ has order $3$.
\end{proposition}

\begin{proof}
The claim is clear if $S$ is $E_8$-extra-special, hence we can apply Lemma \ref{trivial2} and find that $g$ acts non-trivially on all genus one pencils.  Since the order of $g$ is prime to $p$, it  fixes two members $F_1,F_2$ of the pencil, one of which is a double fiber. Since all other fibers are moved,  the set of fixed points $S^g$ is contained in $F_1\cup F_2$. Applying the Lefschetz fixed-point formula, we obtain
\beq\label{eul}
e(S) = 12 = e(S^g) = e(F_1^g)+e(F_2^g), 
\eeq
where $e()$ denotes the $l$-adic topological Euler-Poincar\'e characteristic. 

If one of the fibers, say $F_1$ is smooth, then, since $g$ has odd order and $e(F_2^g) \leq 10$, $\sigma$ acts as an automorphism of order $3$ on $F_1$. Hence, by Proposition \ref{autoelliptic}, $g$ has three fixed points on $F_1$. Therefore, $F_2$ is of type $\tilde{A}_8$, $\tilde{D}_7$ or $\tilde{E}_7$ and $g$ has order $3$. By \cite{Lang1}, we get the last three configurations of the Proposition.

If both fibers or the corresponding half-fibers are singular curves, then $e(F_i) = e(F_i^g)$. Indeed, for irreducible and singular curves, this follows from $e(F_2^g) \leq 10$ and for reducible fibers, this is Lemma \ref{lefschetz} for automorphisms of odd order. The formula for the Euler-Poincar\'e characteristic of an elliptic surface from \cite{CD}, Proposition 5.1.6  implies that $F_1$ and $F_2$ are the only singular fibers of $|D|$. In this case, denoting the number of irreducible components of $F_i$ by $m_i$, we have $m_1+m_2 \geq 8$, hence $|2F|$ is extremal and both fibers are of additive type. The classification of singular fibers of extremal rational genus one fibrations is known \cite{Lang1}, \cite{Lang2}, \cite{Ito}. Since the types of singular fibers of a genus one fibration and of its Jacobian fibration are the same, it is straightforward to check that the list given in the Proposition is complete. 
\end{proof}

\begin{corollary}\label{finiteaut}
If $S$ admits an automorphism $g \in \Aut_{\ct}(S)$ of odd order at least $5$, then $S$ is one of the surfaces in Table \ref{table3}.
\end{corollary}

\begin{proof}
By Proposition \ref{possiblecases}, every genus one fibration on $S$ is extremal. It is shown in \cite{KKM} Section $12$, that such an Enriques surface has finite automorphism group. Using the list of Proposition \ref{possiblecases}, the claim follows from the classification of supersingular Enriques surfaces with finite automorphism group.
\end{proof}

\begin{proposition}\label{order3aut}
Assume that $S$ is not one of the surfaces in Table \ref{table3}. If $S$ admits an automorphism $g \in \Aut_{\ct}(S)$ of order $3$, then $S$ contains the following diagram of $(-2)$-curves

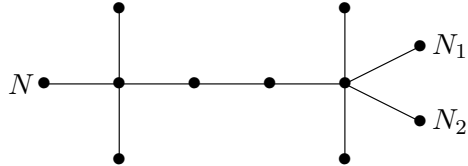
\centerline{
\xy
(0,15)*{};
@={(0,0),(10,0),(20,0),(30,0),(40,0),(50,5),(50,-5),(10,-10),(10,10),(40,10),(40,-10)}@@{*{\bullet}};
(0,0)*{};(40,0)*{}**\dir{-};
(50,5)*{};(40,0)*{}**\dir{-};
(50,-5)*{};(40,0)*{}**\dir{-};
(10,10)*{};(10,-10)*{}**\dir{-};
(40,10)*{};(40,-10)*{}**\dir{-};
(-3,0)*{N};(54,5)*{N_1};(54,-5)*{N_2};
\endxy
 }
 
\noindent In this case, $\Aut_{\ct}(S) = \bbZ/3\bbZ$.
\end{proposition}

\begin{proof}
If every special genus one fibration on $S$ is extremal, then $S$ has finite automorphism group by \cite{KKM} Section $12$. Therefore, we observe that, by Proposition \ref{possiblecases}, $S$ has to admit a special genus one fibration with special bisection $N$ such that $g$ fixes two fibers $F_1$ and $F_2$, where $F_1$ is a smooth supersingular elliptic curve and $F_2$ is of type $\tilde{E}_7$ or $\tilde{D}_7$. If $F_1$ is a simple fiber, then $N$ meets two distinct points of $F_1$, since $g$ does not fix the tangent line at any point of $F_1$. But then, $g$ fixes three points on $N$, hence it fixes $N$ pointwise, which contradicts Lemma \ref{trivial2}.

Therefore, $F_1$ is a double fiber and an argument similar to the above also shows that $N$ meets a component of multiplicity $2$ of $F_2$. Now, depending on the intersection behaviour of $N$ with $F_2$, we see that $N$ and components of $F_2$ form a half-fiber of some other genus one fibration of type $\tilde{D}_4$ or $\tilde{D}_5$ if $F_2$ is of type $\tilde{D}_7$ and of type $\tilde{D}_6$ or $\tilde{E}_6$ if $F_2$ is of type $\tilde{E}_7$. Using the list of Proposition \ref{possiblecases}, we see that only the first and the last case is possible. In the last case, however, one can easily check that the resulting surface will be an exceptional Enriques surface of type $\tilde{E}_6$ (see also \cite[Theorem A]{EkeShep}). Since we assumed that $S$ is not one of the surfaces in Table \ref{table3}, we conclude that $F_2$ is of type $\tilde{D}_7$ and $N$ intersects $F_2$ as follows:

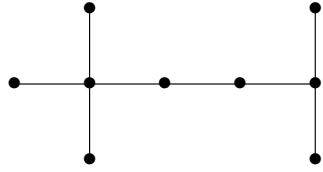
\centerline{
\xy
(0,15)*{};
@={(0,0),(10,0),(20,0),(30,0),(40,0),(10,-10),(10,10),(40,10),(40,-10)}@@{*{\bullet}};
(0,0)*{};(40,0)*{}**\dir{-};
(10,10)*{};(10,-10)*{}**\dir{-};
(40,10)*{};(40,-10)*{}**\dir{-};
\endxy
 }
 
The five leftmost vertices form a diagram of type $\tilde{D}_4$. By Proposition \ref{possiblecases}, this diagram is a half-fiber of a fibration with singular fibers $\tilde{D}_4$ and $\tilde{D}_4$. Adding the second fiber to the diagram, we arrive at the diagram of the Proposition.

Now, observe that the fibration we started with has three $(-2)$-curves as bisections. They are the curves $N,N_1,N_2$ in the diagram from the assertion of the proposition. All of them are fixed pointwise by any cohomologically trivial automorphism of order $2$, since such an automorphism fixes their intersection with $F_1$ and $F_2$. Hence, no such automorphism can exist by Proposition \ref{autoelliptic} applied to a general fiber of the fibration. Since no cohomologically trivial automorphisms of higher order can occur on $S$ by Corollary \ref{finiteaut} and $\Aut_{\ct}(S)$ is cyclic by Lemma \ref{trivial2}, we obtain $\Aut_{\ct}(S) = \bbZ/3\bbZ$.
\end{proof}

\begin{remark}
In fact, one can show that the only genus one fibrations on the supersingular Enriques surface of Proposition \ref{order3aut} are quasi-elliptic fibrations with singular fibers of types $\tilde{D}_4$ and $\tilde{D}_4$ or elliptic fibrations with a unique singular fiber of type $\tilde{D}_7$.
\end{remark}

\begin{theorem}\label{exceptions}
Assume that the automorphism groups of surfaces of type $\tilde{E}_8,\tilde{D}_8,\tilde{E}_7^2$ and $\tilde{E}_6$, are as in Table \ref{exceptnun1} and Table \ref{table3}. Then, for any supersingular Enriques surface $S$ in characteristic $2$, we have
$\Aut_{\ct}(S) \in \{1, \bbZ/2\bbZ, \bbZ/3\bbZ, \bbZ/5\bbZ, \bbZ/7\bbZ, \bbZ/11\bbZ, Q_8 \}$,
\end{theorem}

\section{Numerically trivial automorphisms}
If $K_S = 0$, $\Aut_{\nt}(S) = \Aut_{\ct}(S)$, so we only have to treat the case that $K_S \neq 0$, i.e. $S$ is classical.

By definition, any $g\in \Aut_{\nt}(S)$ leaves invariant any genus one fibration, however, it may act non-trivially on its base, or equivalently, it may act non-trivially on the corresponding pencil $|D|$. Also, by definition, any $g\in \Aut_{\ct}(S)$ fixes the half-fibers of a genus one fibration (their difference in $\NS(S)$ is equal to $K_S$).
The following lemma proves the converse.

\begin{lemma}\label{L7.2.4} A numerically trivial automorphism $g$ that fixes  all half-fibers on $S$ is cohomologically trivial. 
\end{lemma}  

\begin{proof} Since $g$ is numerically trivial, it fixes any smooth rational curve, because they are the unique representatives in $\NS(S)$ of their classes in $\Num(S)$. By assumption, it fixes the linear equivalence class of all irreducible genus one curves. Applying Enriques Reducibility Lemma from \cite{CD}, Corollary 3.2.3 we obtain that $g$ fixes the linear equivalence classes of all curves on $S$.
\end{proof}

\begin{lemma} \label{cyclic}
Let $G$ be a finite, tame group of automorphisms of an irreducible  curve $C$ fixing a nonsingular point $x$. Then, $G$ is cyclic.
\end{lemma}

\begin{proof} Since $G$ is finite and tame, one can linearize the action in the formal neighborhood of the point $x$. It follows that the action of $G$ on the tangent space of $C$ at $x$ is faithful. Since $x$ is nonsingular, the tangent space is one-dimensional and therefore the group is cyclic.
\end{proof}

\begin{theorem}\label{numodd}
Let $S$ be an Enriques surface and assume that $p \neq 2$. Then, $\Aut_{\nt}(S) \cong \bbZ/2^a\bbZ$ with $a \leq 2$. Moreover, if $S$ is unnodal, then $\Aut_{\nt}(S) = \{1\}$.
\end{theorem}

\begin{proof}
By Theorem \ref{cohmain} and Lemma \ref{L7.2.4}, any non-trivial $g \in \Aut_{\nt}(S)$ has order $2$ or $4$, so it suffices to show that $\Aut_{\nt}(S)$ is cyclic. Since $\Aut_{\nt}(S)$ is tame, every numerically trivial automorphism has smooth fixed locus.

Assume that there is some $g \in \Aut_{\nt}(S) \setminus \Aut_{\ct}(S)$. Then, $g$ switches the half-fibers of some elliptic fibration $|2F_1|$ on $S$ by Lemma \ref{L7.2.4}. The argument with the Euler-Poincar\'e characteristics from the proof of Proposition \ref{possiblecases} applies and shows that one of the two fibers $F',F''$ of $|2F_1|$ fixed by $g$, say $F'$, has at least $5$ components. Hence, if $S$ is unnodal, then $\Aut_{\nt}(S) = \{1\}$ by Theorem \ref{cohmain}.

If $F'$ is additive, then it has some component $R$, which is fixed pointwise by $\Aut_{\nt}(S)$, because it is adjacent to at least three other components. Since the fixed loci are smooth, any automorphism fixing a $(-2)$-curve adjacent to $R$ is trivial. Hence, the claim follows from Lemma \ref{cyclic}.

If $F'$ is multiplicative, the fixed point formula shows that $F'$ is of type $\tilde{A}_7$ and $g$ has four fixed points on $F''$. Extend $F_1$ to a non-degenerate $U_{[2]}$-sequence $(F_1,F_2)$. Since $F'.F_2 = 2$, $F'$ contains a cycle of $3$ $(-2)$-curves contained in a fiber $D$ of $|2F_2|$. Now, as in the additive case, we find a $(-2)$-curve, which is fixed pointwise by $\Aut_{\nt}(S)$. Indeed, if $D$ is additive, we use the same argument as before and if $D$ is multiplicative, then some component of $D$ meets a component of $F'$ exactly once in a nonsingular point of $F'$. This component is fixed pointwise by $\Aut_{\nt}(S)$.
\end{proof}

\begin{remark}\label{D4D4}
The previous Theorem is not true if $p = 2$. Indeed, there is an Enriques surface $S$ of type $\tilde{D}_4 + \tilde{D}_4$ with the dual graph of $(-2)$ curves  

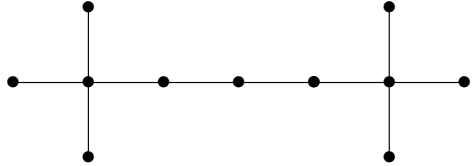
\centerline{
\xy
(0,15)*{};
@={(0,0),(10,0),(20,0),(30,0),(40,0),(10,-10),(10,10),(40,0),(50,10),(50,0),(50,-10),(60,0)}@@{*{\bullet}};
(0,0)*{};(60,0)*{}**\dir{-};
(10,10)*{};(10,-10)*{}**\dir{-};
(50,10)*{};(50,-10)*{}**\dir{-};
\endxy
 }
 
\noindent that satisfies $\Aut_{\nt}(S) = (\bbZ/2\bbZ)^2$ (see \cite{KKM}). Moreover, we have seen in the proof of Corollary \ref{mainclassical} that $\Aut_{\ct}(S) = \{1\}$.
\end{remark}

If $p = 2$, even though we still have the same bound on the size of $\Aut_{\nt}(S)$, the cyclic group of order $4$ can not occur.

\begin{theorem}\label{numeven}
Let $S$ be a classical Enriques surface in characteristic $2$ which is not $E_8$-extra-special. Then, $\Aut_{\nt}(S) \cong (\bbZ/2\bbZ)^b$ with $b \leq 2$.
\end{theorem}

\begin{proof}
By Corollary \ref{mainclassical}, $\Aut_{\ct}(S) \neq \{1\}$ if and only if $S$ is $D_8$-extra-special and for such a surface we have $\Aut_{\nt}(S) = \Aut_{\ct}(S) = \bbZ/2\bbZ$ . Hence, we can assume $\Aut_{\ct}(S) = \{1\}$. By Lemma \ref{L7.2.4}, we have $\Aut_{\nt}(S) = (\bbZ/2\bbZ)^b$ and we have to show $b \leq 2$.
Suppose that $b \geq 3$ and take some half-fiber $F_1$. By Theorem \ref{diagrams2}, we can extend $F_1$ to a non-degenerate $U_{[2]}$-sequence $(F_1,F_2)$. Since $|\Aut_{\nt}(S)| > 4$, there is some numerically trivial involution  $g$ that preserves $F_1$ and $F_2$. By Remark \ref{explicit}, such an automorphism acts trivially on $\sfD_1$, hence it has to coincide with the bielliptic involution associated to $|2F_1 + 2F_2|$. 
Both fibrations have a unique reducible fiber $F$ (resp. $F'$) which has to be simple, since there is some numerically trivial involution which does not preserve $F_i$. By Corollary \ref{extremal}, $F$ and $F'$ are additive and share $8$ components. This is only possible if they are of type $\tilde{D}_8$ or $\tilde{E}_8$. Note that $F.F' = 4$ is impossible if both of them are of type $\tilde{D}_8$. In the remaining cases, it is easy to check that the surface is $D_8$-extra-special. We have already treated this surface.
\end{proof}

 \end{document}